\documentclass[11pt]{amsart}
\usepackage{amsmath,amssymb,amsthm,mathrsfs}
\usepackage{hyperref}
\usepackage[margin=1.1in]{geometry}
\newtheorem{theorem}{Theorem}[section]
\newtheorem{proposition}[theorem]{Proposition}
\newtheorem{lemma}[theorem]{Lemma}
\newtheorem{corollary}[theorem]{Corollary}
\theoremstyle{remark}
\newtheorem{remark}[theorem]{Remark}
\newcommand{\g}{\mathfrak g}
\newcommand{\kalg}{\mathfrak k}
\newcommand{\zalg}{\mathfrak z}
\newcommand{\salg}{\mathfrak s}
\newcommand{\Ad}{\operatorname{Ad}}
\newcommand{\tr}{\operatorname{tr}}

\newcommand{\C}{\mathbb C}
\newcommand{\R}{\mathbb R}

\title[Tam's star-shapedness conjecture]{On Tam's star-shapedness conjecture for numerical ranges of compact adjoint orbits}

\author{Yongzhi Luan}
\address{School of Mathematics, Shandong University, China}
\email{luanyongzhi@email.sdu.edu.cn}

\author{Tin-Yau Tam}
\address{Department of Mathematics and Statistics, University of Nevada, Reno, USA}
\email{ttam@unr.edu}

\date{\today}

\begin{document}
\begin{abstract}
We settle Tam's star-shapedness conjecture for numerical ranges associated with compact adjoint orbits of complex reductive Lie algebras, including its original formulation in terms of the inner product $B_\theta$. Let $\g$ be a complex semisimple Lie algebra, and let $K$ be a connected compact Lie group whose Lie algebra is a compact real form of $\g$. For any $x\in\g$ and any complex linear functional $f\in\g^*$, we prove that $\{f(\Ad(k)x):k\in K\}$ is star-shaped about zero. In the reductive case, the star center is given by the value of $f$ on the central component of $x$, which coincides with the $B_\theta$-pairing of the central components in Tam's original formulation. Our proof is uniform for all complex semisimple Lie algebras. Indeed, using the adjoint Casimir operator, we show that all coordinate functions of the orbit map share a common Laplace eigenvalue; we then use a nodal-hypersurface argument to prove star-shapedness for complex-valued functions in a single eigenspace. The geometric argument relies on a two-sided hypersurface-separation property, which holds whenever the first real cohomology group vanishes. In particular, we do not need simple connectivity or real analyticity. More generally, for finite-dimensional representations of compact connected semisimple Lie groups, we establish star-shapedness whenever the nonfixed component of the orbit vector belongs to a single positive Casimir eigenspace.
\end{abstract}
\maketitle

\section{Introduction and main results}
Let $\g$ be a complex semisimple Lie algebra, let $\kalg$ be a compact real form of $\g$, and let $K$ be a connected compact Lie group with Lie algebra $\kalg$. For $x\in\g$ and $f\in\g^*$, we put
\[
W_f(x):=\{f(\Ad(k)x):k\in K\}\subseteq\C.
\]
Let $B(\cdot, \cdot)$ denote the Killing form on the semisimple part of $\g$, and let
$\theta$ be the conjugation of $\g$ with respect to the compact real form $\kalg$.
Then the inner product
\[
 B_\theta(a,b):=-B(a,\theta b)
\]
is linear in its first argument and conjugate-linear in its second argument. In the reductive case, we equip the center
with any positive-definite $K$-invariant inner product and take its
orthogonal direct sum with the above form on $[\g,\g]$; we denote the
resulting inner product on $\g$ again by $B_\theta$. For $x,y\in\g$, we set
\[
 W_x(y):=\{B_\theta(x,\Ad(k)y):k\in K\}.
\]
In 2001, Tam formulated the star-shapedness conjecture
\cite[Conjecture~5.1]{Tam2001}: for a complex reductive
Lie algebra $\g=\salg\oplus\zalg$ with decompositions
$x=x_{\salg}+x_{\zalg}$ and $y=y_{\salg}+y_{\zalg}$,
the set $W_x(y)$ is star-shaped about $B_\theta(x_{\zalg},y_{\zalg})$.
Later, \DJ okovi\'c and Tam
\cite[Conjecture~3.8]{DjokovicTam} stated an equivalent formulation in the semisimple
case using an arbitrary complex linear functional, namely that
$f(\Ad(K)x)$ is star-shaped about zero for every $x\in\g$ and
$f\in\g^*$ (see also \cite[Conjecture~2.10]{TamCheung}).
Indeed, every functional $f\in\g^*$ can be written uniquely as
$f(a)=B_\theta(a,y)$ for some $y\in\g$, and by the conjugate symmetry of $B_\theta$,
interchanging the two arguments does not affect star-shapedness.

The conjecture was motivated by the 1996 result of Cheung and Tsing \cite{CheungTsing}, who proved that the $C$-numerical range of any $n\times n$ complex matrix $A$ is star-shaped about $(\operatorname{tr}A)(\operatorname{tr}C)/n$. This result naturally led to the question of whether a similar star-shapedness property holds for numerical ranges associated with compact adjoint orbits of complex reductive Lie algebras, which Tam formulated as a conjecture in 2001 \cite{Tam2001}.

Subsequently, several partial results were obtained for individual Dynkin types.
In 2003, \DJ okovi\'c and Tam \cite[Theorem~3.11]{DjokovicTam} proved the conjecture for simply laced types without $E_8$ factors. In particular, their result settled \emph{types $D$, $E_6$, and $E_7$}, and also recovered type $A$.
In 2011, Cheung and Tam \cite[Theorem~4.8]{CheungTam} settled \emph{type $B$}, and their method also applies to type $D$. In addition, they proved a distinct vector-valued result: the orthogonal projection (with respect to the Killing form) of any compact adjoint orbit onto $\mathfrak h=\mathfrak t\oplus i\mathfrak t$, where $\mathfrak t$ is a maximal abelian subalgebra of $\kalg$, is star-shaped about zero \cite[Theorem~3.1]{CheungTam}.
However, \emph{type $C_n$ for $n\ge3$, as well as the exceptional types $E_8$, $F_4$, and $G_2$}, remained open. In this paper, we settle the conjecture for all complex semisimple Lie algebras by a single unified argument, and we recover its original reductive formulation with the prescribed affine star center.

Our approach here differs substantially from the earlier proofs. While Cheung and Tsing worked with unitary matrix orbits, \DJ okovi\'c and Tam as well as Cheung and Tam used Lie-theoretic reductions tailored to specific root systems. By contrast, we neither invoke the classification of root systems nor treat individual Dynkin types case by case. Instead, using the adjoint Casimir operator, we show that all coordinate functions of the orbit map lie in a single Laplace eigenspace. We then apply a geometric argument based on nodal hypersurfaces, their separation properties, and Green's identity to obtain the full radial segments required for star-shapedness. In this way, we get a uniform proof for all complex semisimple Lie algebras, and the reductive case follows immediately by translation.

\begin{theorem}[Semisimple star-shapedness]\label{thm:semisimple}
Let $\g$ be a complex semisimple Lie algebra, let $\kalg$ be a compact real form of $\g$, and let $K$ be a connected compact Lie group with Lie algebra $\kalg$, acting on $\g$ by the complex-linear extension of its adjoint representation. Then for every $x\in\g$ and $f\in\g^*$, one has that
\[
tW_f(x)\subseteq W_f(x)\qquad(0\leq t\leq1).
\]
\end{theorem}

In the original reductive setting \cite[Conjecture~5.1]{Tam2001}, the star center is a natural \emph{affine} point. Let us write a complex reductive Lie algebra as $\g=\zalg\oplus\salg$, where $\zalg=Z(\g)$ is the center and $\salg=[\g,\g]$ is the semisimple derived subalgebra, and decompose $x\in\g$ as $x=x_{\zalg}+x_{\salg}$.

\begin{corollary}\label{cor:reductive}
Let $\g$ be a complex reductive Lie algebra and let $K$ be a connected compact Lie group whose Lie algebra is a compact real form of $\g$, acting on $\g$ by the adjoint representation. Then for every $x\in\g$ and $f\in\g^*$, the set $W_f(x)$ is star-shaped about $f(x_{\zalg})$. More precisely, we have
\[
tW_f(x)+(1-t)f(x_{\zalg})\subseteq W_f(x)
\quad(0\leq t\leq1).
\]
Equivalently, for any $a,b\in\g$, the set
$W_a(b)=\{B_\theta(a,\Ad(k)b):k\in K\}$ is star-shaped
about $B_\theta(a_{\zalg},b_{\zalg})$, as conjectured in
\cite[Conjecture~5.1]{Tam2001}.
\end{corollary}

Our proof of Theorem~\ref{thm:semisimple} is based on a geometric theorem for a single Laplace eigenspace. It is worth noting that Gichev's theorem on common zeros \cite[Theorem~1]{Gichev}, combined with our adjoint Casimir calculation in Section 4, already implies that $0\in W_f(x)$ (see also \cite[Corollary~2]{Gichev}). To go further, we build on Gichev's argument using Green's identity argument for regular nodal curves on the two-sphere \cite[Lemma~3 and Proposition~1]{Gichev}, and combine it with a perturbation argument to obtain the entire radial segments. In Section 3, we establish this geometric principle on closed manifolds satisfying a two-sided hypersurface-separation property; in particular, a sufficient topological condition is the vanishing of the first de Rham cohomology group $H^1(M;\R)$, as in the setting of Gichev's theorem on common zeros.

\section{Preliminaries on adjoint orbits and numerical ranges}

Let us briefly recall the basic setup of the conjecture. Every complex
semisimple Lie algebra $\g$ admits a compact real form $\kalg$, so that
$\g=\kalg\oplus i\kalg$ as a direct sum of real vector spaces. Let $K$ be a
connected compact Lie group with Lie algebra $\kalg$. Then the adjoint action of $K$
on $\g$ is complex-linear and preserves the complex Lie bracket.
For any $x\in\g$, the adjoint orbit
\[
\mathcal O_x:=\Ad(K)x=\{\Ad(k)x:k\in K\}
\]
is compact and connected. Since the numerical range $W_f(x)$ is the image
of $\mathcal O_x$ under the complex linear functional $f$, it is
a compact connected subset of $\C$, though in general it need
not be convex. Note that connectedness alone does not guarantee
star-shapedness: for instance, the unit circle is connected, but it does not contain the radial
segments joining its points to the origin.

Note also that the choice of the connected compact Lie group $K$ with Lie algebra $\kalg$
does not affect the orbit $\mathcal O_x$. Indeed, the adjoint action of $K$ factors
through the adjoint group $\operatorname{Int}(\kalg)$, so any two connected compact Lie groups with Lie algebra
$\kalg$ have the exact same image in $\operatorname{Aut}(\g)$.
Therefore, in the proof of Theorem~\ref{thm:semisimple}, we may pass to the simply connected covering
group $\widetilde K$ without changing the numerical range $W_f(x)$.

Let $B$ be the complex Killing form of $\g$, and define the conjugation
$\theta(X+iY):=X-iY$ for $X,Y\in\kalg$. Then $\theta$ is a
conjugate-linear automorphism of $\g$ as a real Lie algebra.
Since $B$ is negative definite on $\kalg$, the form
\[
B_\theta(a,b):=-B(a,\theta b)
\]
is a inner product on $\g$, linear in the first variable and
conjugate-linear in the second. Since $\theta$ commutes with $\Ad(k)$ for all $k\in K$, one has that
\[
B_\theta(\Ad(k)a,\Ad(k)b)=B_\theta(a,b),\qquad k\in K.
\]
This is the inner product used in Tam's original formulation
\cite{Tam2001}. Since $B_\theta$ is nondegenerate, we can use it to identify $\g$ conjugate-linearly with its complex dual $\g^*$. Namely,
for every complex linear functional $f\in\g^*$, there exists a unique element $y\in\g$
such that $f(a)=B_\theta(a,y)$ for all $a\in\g$.
Consequently, the numerical ranges $W_f(x)$ defined via linear functionals $f\in\g^*$ and
the sets $W_x(y)$ defined via $B_\theta$ describe the exact same family of subsets
of $\C$.

When $\g=\zalg\oplus\salg$ is a complex reductive Lie algebra, the Killing form
vanishes on the center $\zalg$ and thus does not give an inner product on
all of $\g$. Instead, we take the orthogonal direct sum of $B_\theta$ on $\salg$ with any
inner product on $\zalg$. Since $K$ is connected, it acts trivially on $\zalg$.
Therefore, for $x=x_\zalg+x_\salg$, we have
\begin{equation}\label{eq:affine-decomposition}
W_f(x)=f(x_\zalg)+W_{f|_\salg}(x_\salg).
\end{equation}
From \eqref{eq:affine-decomposition}, one sees immediately why the star center in the reductive case
is generally nonzero, and why the semisimple case in Theorem~\ref{thm:semisimple} is the main step.

Recall that a subset $S\subseteq\C$ is \emph{star-shaped} about
$c\in S$ if
\[
tz+(1-t)c\in S\qquad\text{for all }z\in S\text{ and }0\leq t\leq1.
\]
Thus, the property $0\in W_f(x)$ is much weaker than
star-shapedness about zero, as the latter requires the \emph{entire} radial
segment from $0$ to $f(\Ad(k)x)$ to lie in $W_f(x)$ for every $k\in K$. In our proof,
we first use Green's identity to find zeros on connected components of regular nodal hypersurfaces,
and then use the connectedness of these components to obtain all intermediate values along the radial segment.
For this reason, we work with connected components of regular nodal sets rather than
merely applying existence theorems for common zeros.

Let us summarize the three main steps of our proof. First, in Section 3, we prove a general geometric principle:
if all coordinate functions of a nonvanishing smooth vector-valued map on a
closed Riemannian manifold $M$ lie in a single positive Laplace eigenspace,
and every connected two-sided hypersurface separates $M$, then every complex-linear projection of the image is star-shaped about zero. Second, in Section 4, for
a compact semisimple Lie group $K$ equipped with its bi-invariant Killing-form metric,
we use the adjoint Casimir operator to show that all coordinate functions of the orbit map
share the same Laplace eigenvalue $\lambda=1$. Third, by passing to the universal covering group $\widetilde K$, we guarantee the hypersurface-separation property and complete the proof of Theorem~\ref{thm:semisimple}; Corollary~\ref{cor:reductive} then follows from
\eqref{eq:affine-decomposition}. None of these steps uses the
classification of simple Lie algebras.

\section{A geometric star-shapedness principle}
Throughout this section, all manifolds are smooth Riemannian manifolds without boundary, and a closed manifold means a compact manifold without boundary. We write the Laplace--Beltrami operator with the sign convention $\Delta v:=\operatorname{div}(\nabla v)$, where $\nabla v$ is the Riemannian gradient of $v$ and $\operatorname{div}$ is the divergence. Under this convention, the equation $\Delta v=-\lambda v$ corresponds to an eigenvalue $\lambda$ of $-\Delta$. We say that a connected closed manifold $M$ has the \emph{two-sided separation property} if every nonempty connected, compact, embedded, two-sided hypersurface $C\subset M$ without boundary separates $M$ into two connected components. Here, ``two-sided'' means that the normal bundle of $C$ in $M$ is trivial, i.e., $C$ admits a global unit normal vector field.

\begin{lemma}[A sufficient topological condition]\label{lem:separation}
Let $M$ be a connected closed smooth manifold whose first de Rham cohomology group $H^1(M;\R)$ vanishes. Then $M$ has the two-sided separation property. In particular, every closed simply connected manifold has this property. Moreover, for any hypersurface $C\subset M$ as above, the closure of each connected component of $M\setminus C$ is a compact smooth manifold with boundary $C$.
\end{lemma}
\begin{proof}
We fix a global unit normal vector field along $C$, which gives $C$ a coorientation. Suppose on the contrary that $M\setminus C$ is connected. Let us choose a short smooth arc in $M$ that crosses $C$ transversely at a single point in the positive normal direction, with its two endpoints lying on opposite sides of $C$. Since $M\setminus C$ is connected, we can join these two endpoints by a piecewise smooth path inside $M\setminus C$. After smoothing the concatenated path at the junctions, we get a smooth closed loop $\gamma$ in $M$ that intersects $C$ transversely at a single point with signed intersection number $+1$.

On the other hand, using the coorientation of $C$, we can define the integer-valued signed intersection number with $C$ for any closed loop transverse to $C$ by summing the signs of all crossings. Since crossings are created or canceled in pairs of opposite signs under smooth homotopies, this intersection number is a homotopy invariant and thus induces a group homomorphism $H_1(M;\mathbb Z)\to\mathbb Z$. Since the loop $\gamma$ has intersection number $+1$, this homomorphism is nonzero. By the universal coefficient theorem, it gives a nonzero cohomology class in $H^1(M;\R)$, which contradicts the assumption $H^1(M;\R)=0$. Note that we do not need $M$ to be orientable for this argument; a coorientation of $C$ suffices.

Therefore, $M\setminus C$ is disconnected. Now take a tubular neighborhood of $C$ diffeomorphic to $C\times(-\varepsilon,\varepsilon)$. Removing $C$ leaves two connected sides $C\times(-\varepsilon,0)$ and $C\times(0,\varepsilon)$. Since $C$ is connected, each side lies in a single connected component of $M\setminus C$; meanwhile, every connected component of $M\setminus C$ must intersect this tubular neighborhood and hence must contain at least one of the two sides. Thus $M\setminus C$ has at most two connected components. Since $M\setminus C$ is disconnected, it has exactly two connected components, and the closure of each component is a compact smooth manifold with boundary equal to $C$. This finishes the proof.
\end{proof}

\begin{lemma}[Green's identity on a nodal component]\label{lem:green}
Suppose that $M$ is a connected closed Riemannian manifold with the two-sided separation property. Let $u,v$ be smooth real-valued functions on $M$ satisfying
\[
\Delta u=-\lambda u,\qquad \Delta v=-\lambda v
\]
for the same real number $\lambda$. If $0$ is a regular value of $v$, then for every connected component $C$ of $v^{-1}(0)$, the image $u(C)$ is a compact interval containing zero.
\end{lemma}
\begin{proof}
Since $0$ is a regular value of $v$, the zero locus $v^{-1}(0)$ is a compact embedded hypersurface in $M$, and each connected component $C$ admits the global unit normal vector field $\nabla v/|\nabla v|$. By the two-sided separation property, we can choose a connected component $D$ of $M\setminus C$ whose closure $\overline D$ has smooth boundary $C$. Let $\nu$ denote the outward unit normal along $C=\partial D$. By Green's identity on $\overline D$, we have
\[
0=\int_D(u\Delta v-v\Delta u)\,dV
=\int_C(u\partial_\nu v-v\partial_\nu u)\,dS
=\int_Cu\,\partial_\nu v\,dS.
\]
Here $dV$ and $dS$ denote the Riemannian volume measure on $M$ and the induced surface measure on $C$, respectively, so the Divergence Theorem and Green's identity hold even if $M$ is non-orientable. Since $C$ is connected and $\nabla v$ never vanishes on $C$, the normal derivative $\partial_\nu v$ is either strictly positive ($+|\nabla v|$) everywhere on $C$ or strictly negative ($-|\nabla v|$) everywhere on $C$. If $u$ had no zero on $C$, then $u$ would also have a constant nonzero sign on $C$, so the integral $\int_C u\,\partial_\nu v\,dS$ would be nonzero, a contradiction. Therefore, we have $0\in u(C)$. Since $C$ is compact and connected and $u$ is continuous, it follows that $u(C)$ is a compact interval containing zero.
\end{proof}

\begin{lemma}[Local sign change]\label{lem:sign}
Let $M$ be a connected smooth Riemannian manifold, and let $v$ be a nonzero smooth real-valued function on $M$ satisfying $\Delta v=-\lambda v$ with $\lambda>0$. Then every neighborhood of any zero of $v$ contains both a point where $v>0$ and a point where $v<0$.
\end{lemma}
\begin{proof}
Suppose that $v(p)=0$ and $v\geq0$ on a connected coordinate ball $U$ centered at $p$. Then we have $\Delta v=-\lambda v\leq0$ on $U$. By the strong minimum principle for superharmonic functions, we see that $v$ must vanish identically on $U$. By Aronszajn's unique continuation theorem for second-order elliptic equations with smooth coefficients \cite{Aronszajn}, it follows that $v\equiv0$ on the entire connected manifold $M$, which contradicts our assumption that $v$ is nonzero. Applying the same argument to $-v$ on arbitrarily small coordinate balls around $p$, we see that $v$ takes both positive and negative values in every neighborhood of $p$.
\end{proof}

\begin{lemma}[Generic regular zero levels]\label{lem:generic}
Let $M$ be a closed smooth manifold, let $V$ be a finite-dimensional complex vector space, and let $a:M\to V\setminus\{0\}$ be a smooth map. Then the set of complex linear functionals $\ell\in V^*$ for which $0$ is a regular value of $\operatorname{Im}(\ell\circ a)$ is dense in $V^*$.
\end{lemma}
\begin{proof}
Let us consider the smooth map $\Phi:V^*\times M\to\R$ defined by $\Phi(\ell,p):=\operatorname{Im}\ell(a(p))$, where we view $V^*$ as a real vector space. Since $a(p)\ne0$ for all $p\in M$, the partial differential of $\Phi$ in the $V^*$-direction is surjective at every point $(\ell,p)\in V^*\times M$. Consequently, $0$ is a regular value of $\Phi$, and $N:=\Phi^{-1}(0)$ is a smooth submanifold of $V^*\times M$.

Now consider the projection $\pi:N\to V^*$ given by $\pi(\ell,p)=\ell$, and write $v_\ell(p):=\operatorname{Im}\ell(a(p))$. At any point $(\ell,p)\in N$, the tangent space of $N$ is given by
\[
T_{(\ell,p)}N=\bigl\{(\eta,X)\in V^*\times T_pM:
\operatorname{Im}\eta(a(p))+d_pv_\ell(X)=0\bigr\}.
\]
Note that the differential $d\pi_{(\ell,p)}:T_{(\ell,p)}N\to V^*$ sends $(\eta,X)$ to $\eta$. If $d_pv_\ell\ne0$, then for every $\eta\in V^*$ we can choose $X\in T_pM$ satisfying $\operatorname{Im}\eta(a(p))+d_pv_\ell(X)=0$, which implies that $d\pi_{(\ell,p)}$ is surjective. Conversely, if $d_pv_\ell=0$, then the image of $d\pi_{(\ell,p)}$ is contained in the proper real hyperplane $\{\eta\in V^*:\operatorname{Im}\eta(a(p))=0\}$, so $d\pi_{(\ell,p)}$ is not surjective. Therefore, $\ell\in V^*$ is a regular value of $\pi$ if and only if $0$ is a regular value of $v_\ell$. Since $\pi:N\to V^*$ is a smooth map between finite-dimensional smooth manifolds, we can apply Sard's theorem \cite{Hirsch} to conclude that the regular values of $\pi$ have full measure and hence are dense in $V^*$. Note that this also covers the case where $v_\ell^{-1}(0)$ is empty, for which $0$ is vacuously a regular value.
\end{proof}

\begin{proposition}[Uniform eigenfunction principle]\label{prop:geometry}
Let $M$ be a connected closed smooth Riemannian manifold with the two-sided separation property, let $V$ be a finite-dimensional complex vector space, and let $a:M\to V\setminus\{0\}$ be a smooth map. Suppose that there exists a constant $\lambda>0$ such that
\[
\Delta(\ell\circ a)=-\lambda(\ell\circ a)\qquad\text{for all }\ell\in V^*.
\]
Then the image $(\ell\circ a)(M)$ is star-shaped about zero for every $\ell\in V^*$.
\end{proposition}
\begin{proof}
We fix a norm on $V$ and equip $V^*$ with the dual norm. Take any functional $\ell\in V^*$, any point $p_0\in M$, and any $t\in[0,1]$, and put $F:=\ell\circ a$. We want to show that $tF(p_0)\in F(M)$. If $F(p_0)=0$, the claim is trivial. Thus we assume $F(p_0)\neq 0$. By multiplying $\ell$ by a unimodular complex number, we may assume without loss of generality that $F(p_0)$ is real and nonzero. Let us write $F=u+iv$ for its real and imaginary parts; in particular, we have $v(p_0)=0$ and $F(p_0)=u(p_0)$.

If $v\equiv0$ on $M$, then $F=u$ is real-valued, and hence $F(M)\subset\R$ is a compact interval because $M$ is connected. Integrating $\Delta F=-\lambda F$ over $M$, we have $-\lambda\int_M F\,dV=\int_M\Delta F\,dV=0$. Since $\lambda>0$, we get $\int_M F\,dV=0$, which implies $0\in F(M)$. Thus the entire closed interval between $0$ and $F(p_0)$ is contained in $F(M)$.

Now we assume that $v\not\equiv0$. Let $\{U_j\}_{j=1}^\infty$ be a sequence of path-connected coordinate balls centered at $p_0$ whose diameters tend to $0$ as $j\to\infty$. Since $v(p_0)=0$, by Lemma~\ref{lem:sign}, we can pick points $r_j^+, r_j^-\in U_j$ such that $v(r_j^+)>0>v(r_j^-)$. Put $b_j:=\min\{v(r_j^+),-v(r_j^-)\}>0$ and $A:=\max_{p\in M}\|a(p)\|>0$. By Lemma~\ref{lem:generic}, we can choose a sequence of functionals $\ell_j\in V^*$ such that $0$ is a regular value of $v_j:=\operatorname{Im}(\ell_j\circ a)$ and
\[
\|\ell_j-\ell\|<\min\bigl\{j^{-1},\,b_j/(2A)\bigr\}.
\]
Write $F_j:=\ell_j\circ a=u_j+iv_j$. Then $F_j$ converges uniformly to $F$ on $M$:
\[
\sup_{p\in M}|F_j(p)-F(p)|\le A\|\ell_j-\ell\|\longrightarrow0\qquad(j\to\infty).
\]
Moreover, by our choice of $\|\ell_j-\ell\|<b_j/(2A)$, we have $v_j(r_j^+)>b_j/2>0$ and $v_j(r_j^-)<-b_j/2<0$. Joining $r_j^+$ and $r_j^-$ by a continuous path inside $U_j$, we find a point $p_j\in U_j$ such that $v_j(p_j)=0$. Since $\operatorname{diam}(U_j)\to 0$, we have $p_j\to p_0$ as $j\to\infty$.

Let $C_j$ be the connected component of the regular level set $v_j^{-1}(0)$ that contains $p_j$. Since $\Delta u_j=-\lambda u_j$ and $\Delta v_j=-\lambda v_j$, it follows from Lemma~\ref{lem:green} that $u_j(C_j)$ is a compact interval containing both $0$ and $u_j(p_j)$. Since $t\in[0,1]$ and $v_j=0$ on $C_j$, we can choose a point $q_j\in C_j$ such that
\[
F_j(q_j)=u_j(q_j)=t\,u_j(p_j)=tF_j(p_j).
\]
Since $M$ is compact, after passing to a subsequence, we may assume that $q_j\to q$ for some point $q\in M$. Letting $j\to\infty$ and using the uniform convergence $F_j\to F$ together with $p_j\to p_0$, we obtain
\[
F(q)=\lim_{j\to\infty}F_j(q_j)=t\lim_{j\to\infty}F_j(p_j)=tF(p_0).
\]
Finally, undoing the initial unimodular rotation of $\ell$, we conclude that $(\ell\circ a)(M)$ is star-shaped about zero. This finishes the proof.
\end{proof}

\begin{remark}[On the hypotheses of Proposition~\ref{prop:geometry}]\label{rem:limits}
Note that Proposition~\ref{prop:geometry} requires neither real analyticity nor simple connectivity of $M$: smoothness of $M$ and $a$ is sufficient (by Aronszajn's elliptic unique continuation theorem), and the only topological requirement is the two-sided separation property. By Lemma~\ref{lem:separation}, the condition $H^1(M;\R)=0$ is sufficient for this separation property; note that $H^1(M;\R)=0$ is strictly weaker than $H^1(M;\mathbb Z/2\mathbb Z)=0$ (for instance, $M=\mathbb{RP}^3$ satisfies $H^1(\mathbb{RP}^3;\R)=0$ while $H^1(\mathbb{RP}^3;\mathbb Z/2\mathbb Z)\neq 0$). In the purely real case ($v\equiv 0$), we do not even need the nonvanishing of $a$ or the separation property, because the image of $M$ is already an interval containing zero. On the other hand, in the complex-valued case, neither the positivity of $\lambda$ nor the topological condition can be dropped: for example, on the unit circle $M=S^1$, the eigenfunction $F(e^{it})=e^{it}$ satisfies $\Delta F=-F$ ($\lambda=1$), but its image is $S^1$, which is not star-shaped; similarly, a nonzero constant function ($\lambda=0$) has a single nonzero point as its image. Here we use the assumption $a(p)\neq 0$ in Lemma~\ref{lem:generic} to guarantee that generic zero levels are regular, though we do not claim this nonvanishing condition is strictly necessary.
\end{remark}

\begin{remark}[Role of the common eigenvalue]
If $u$ and $v$ were eigenfunctions of $-\Delta$ with two distinct
eigenvalues $\lambda_1\neq\lambda_2$, then the volume integral $\int_D(u\Delta v-v\Delta u)\,dV=(\lambda_1-\lambda_2)\int_D uv\,dV$ in Green's identity would not vanish in general. Thus, having a common Laplace eigenvalue for the real and imaginary parts $u$ and $v$ is essential to our argument. In the setting of compact adjoint orbits, this common eigenvalue comes directly from the action of the adjoint Casimir operator, without any case-by-case calculation of roots or weights.
\end{remark}

\section{Proofs of the main results}
Let $B_\kalg$ denote the real Killing form on the compact semisimple Lie algebra $\kalg$, and equip $\kalg$ with the inner product $\langle X,Y\rangle:=-B_\kalg(X,Y)$.

\begin{lemma}[Normalized adjoint Casimir operator]\label{lem:casimir}
Let $E_1,\dots,E_d$ be an orthonormal basis of $\kalg$ with respect to $\langle\cdot,\cdot\rangle=-B_\kalg$. Then we have
\[
\sum_{j=1}^d(\operatorname{ad}E_j)^2=-I_\kalg,
\]
and the same identity holds on the complexification $\g=\kalg\otimes_\R\C$.
\end{lemma}
\begin{proof}
Put $T:=\sum_{j=1}^d(\operatorname{ad}E_j)^2$. Since the Killing form $B_\kalg$ is $\operatorname{ad}$-invariant, each operator $\operatorname{ad}E_j$ is skew-adjoint with respect to $\langle\cdot,\cdot\rangle$. Thus, for any $X,Y\in\kalg$, we have
\[
\langle TX,Y\rangle
=-\sum_{j=1}^d\langle[E_j,X],[E_j,Y]\rangle
=-\sum_{j=1}^d\langle[X,E_j],[Y,E_j]\rangle.
\]
On the other hand, expanding the trace in the definition of $B_\kalg(X,Y)$ with respect to the orthonormal basis $\{E_j\}_{j=1}^d$ and using the skew-adjointness of $\operatorname{ad}X$, we get
\[
B_\kalg(X,Y)=\tr_\R(\operatorname{ad}X\,\operatorname{ad}Y)
=\sum_{j=1}^d\langle(\operatorname{ad}X\,\operatorname{ad}Y)E_j,E_j\rangle
=-\sum_{j=1}^d\langle[Y,E_j],[X,E_j]\rangle.
\]
Comparing the two expressions, we obtain $\langle TX,Y\rangle=B_\kalg(X,Y)=-\langle X,Y\rangle$ for all $X,Y\in\kalg$, which gives $T=-I_\kalg$. Extending $T$ complex-linearly to $\g=\kalg\otimes_\R\C$, we see that the same identity holds on $\g$. Note that this proof applies directly to any compact semisimple Lie algebra without decomposing it into simple ideals.
\end{proof}

\begin{lemma}[Common Laplace eigenvalue]\label{lem:laplace}
Let $K$ be a connected compact semisimple Lie group equipped with the bi-invariant Riemannian metric induced by $-B_\kalg$. Then for any $x\in\g$ and $f\in\g^*$, the function
\[
F_f(k):=f(\Ad(k)x)\qquad\text{satisfies}\qquad\Delta F_f=-F_f.
\]
Moreover, if $x\ne0$, then the smooth orbit map $a_x(k):=\Ad(k)x$ never vanishes on $K$.
\end{lemma}
\begin{proof}
Let $X_1,\dots,X_d$ be the left-invariant vector fields on $K$ corresponding to the orthonormal basis $E_1,\dots,E_d$ of $\kalg$. Since the Riemannian metric on $K$ is bi-invariant, the Laplace--Beltrami operator on smooth functions is given by $\Delta=\sum_{j=1}^d X_j^2$. Differentiating $F_f$ twice along the one-parameter curves $t\mapsto k\exp(tE_j)$, we have
\[
X_j^2F_f(k)=f\bigl(\Ad(k)(\operatorname{ad}E_j)^2x\bigr).
\]
Summing over $j=1,\dots,d$ and applying Lemma~\ref{lem:casimir}, we get $\Delta F_f(k)=-F_f(k)$. Finally, since $\Ad(k)$ is a linear automorphism of $\g$ for every $k\in K$, we see that $a_x(k)=\Ad(k)x\neq 0$ whenever $x\neq 0$. This finishes the proof.
\end{proof}

\begin{proof}[Proof of Theorem~\ref{thm:semisimple}]
If $x=0$ or $\g=0$, the result is immediate. Now assume that $\g\neq 0$ and $x\neq 0$. Let $\widetilde K$ be the universal covering group of $K$. Since $K$ is a connected compact semisimple Lie group, it is a standard fact (see \cite{Knapp}) that $\widetilde K$ is also compact and semisimple. Since the covering homomorphism $\widetilde K\to K$ is surjective and its kernel acts trivially under the adjoint representation on $\g$, replacing $K$ by $\widetilde K$ does not change the numerical range $W_f(x)$.

Since $\widetilde K$ is connected, closed, and simply connected, it follows from Lemma~\ref{lem:separation} that $\widetilde K$ has the two-sided separation property. By Lemma~\ref{lem:laplace}, the orbit map $a_x:\widetilde K\to\g\setminus\{0\}$ is smooth and nonvanishing, and for every $f\in\g^*$ the function $f\circ a_x$ satisfies $\Delta(f\circ a_x)=-(f\circ a_x)$ with the common positive eigenvalue $\lambda=1$. Applying Proposition~\ref{prop:geometry} with $M=\widetilde K$ and $V=\g$, we conclude that $W_f(x)=(f\circ a_x)(\widetilde K)$ is star-shaped about zero. This finishes the proof.
\end{proof}

\begin{proof}[Proof of Corollary~\ref{cor:reductive}]
Since $K$ is connected, the decomposition $\g=\zalg\oplus\salg$ is invariant under $\Ad(K)$, and $K$ acts trivially on the center $\zalg$. Moreover, the connected Lie subgroup $K_{\salg}\subseteq K$ with Lie algebra $\kalg\cap\salg$ is a connected compact semisimple Lie group whose adjoint orbits in $\salg$ coincide with those of $K$. Therefore, for any $x=x_{\zalg}+x_{\salg}\in\g$ and $f\in\g^*$, we have
\[
f(\Ad(k)x)=f(x_{\zalg})+f|_{\salg}(\Ad(k)x_{\salg}),\qquad k\in K.
\]
By Theorem~\ref{thm:semisimple}, the set $\{f|_{\salg}(\Ad(k)x_{\salg}):k\in K\}$ is star-shaped about zero. Since $W_f(x)=f(x_{\zalg})+W_{f|_{\salg}}(x_{\salg})$, we obtain
\[
tW_f(x)+(1-t)f(x_{\zalg})\subseteq W_f(x),\qquad 0\leq t\leq1.
\]

To deduce Tam's original inner-product formulation, for a fixed $b\in\g$ we define the linear functional $f_b\in\g^*$ by $f_b(c):=B_\theta(c,b)$. Since $B_\theta$ is $K$-invariant, we have
\[
B_\theta(a,\Ad(k)b)=B_\theta(\Ad(k^{-1})a,b)=f_b(\Ad(k^{-1})a),\qquad k\in K.
\]
As $k$ runs over the group $K$, so does $k^{-1}$. Thus we have $W_a(b)=W_{f_b}(a)$, which is star-shaped about
$f_b(a_{\zalg})=B_\theta(a_{\zalg},b_{\zalg})$. This finishes the proof.
\end{proof}

Let us illustrate our geometric proof by recovering the Cheung--Tsing theorem.
Let $n\ge2$, take two traceless matrices $A_0,C_0\in\mathfrak{sl}_n(\C)$, and consider the function
\[
F(U):=\tr(C_0UA_0U^*),\qquad U\in SU(n).
\]
We equip $SU(n)$ with the bi-invariant Riemannian metric induced by the negative
Killing form on $\mathfrak{su}(n)$, and let $E_1,\ldots,E_{n^2-1}$ be an
orthonormal basis of $\mathfrak{su}(n)$. By Lemma~\ref{lem:casimir}, the normalized adjoint Casimir identity takes the form
\[
\sum_{j=1}^{n^2-1}[E_j,[E_j,A_0]]=-A_0.
\]
Consequently, we obtain
\[
\Delta F(U)=\tr\!\left(C_0U\Bigl(\sum_{j=1}^{n^2-1}[E_j,[E_j,A_0]]\Bigr)U^*\right)
=-F(U).
\]
To see how the nodal-hypersurface argument works here, take any value $z=F(U_0)$ with $U_0\in SU(n)$. If $z=0$, the
radial segment is trivial. Otherwise, we multiply $F$ by a
unimodular complex scalar so that $F(U_0)$ becomes real and nonzero, and we still denote the
rotated function by $F=u+iv$. Then both $u$ and $v$ satisfy $\Delta u=-u$ and $\Delta v=-v$. Suppose first that $0$ is a regular value of $v$,
and let $N$ be the connected component of $v^{-1}(0)$ containing $U_0$.
Then $N$ is a compact two-sided hypersurface in $SU(n)$, and since $SU(n)$ is simply connected, $N$ separates $SU(n)$ into two components.
Let $D$ be one of the complementary domains with boundary $\partial D=N$. By Green's identity, we have
\[
0=\int_D(u\Delta v-v\Delta u)\,dV
 =\int_N u\,\partial_\nu v\,dS.
\]
Since $\partial_\nu v$ has a constant nonzero sign on $N$, the function $u$
must vanish at some point on $N$. Since $N$ is connected, it follows that $F(N)=u(N)$
is a compact real interval containing both $0$ and $F(U_0)$, and hence contains
the entire line segment joining $0$ and $F(U_0)$. After undoing the rotation, we obtain the radial segment
from $0$ to $z$ for the original function $F$. When $0$ is not a regular
value of $v$, the perturbation argument in the proof of
Proposition~\ref{prop:geometry} yields the same result (and the case $v\equiv 0$ is also covered there).

Now, for arbitrary matrices $A,C\in M_n(\C)$, we decompose them into scalar and traceless parts:
\[
A=A_0+\frac{\tr A}{n}I,\qquad
C=C_0+\frac{\tr C}{n}I,
\qquad \tr A_0=\tr C_0=0.
\]
Since scalar matrices are central, they only shift the numerical range by a constant translation:
\[
\{\tr(CUAU^*):U\in U(n)\}
=\frac{\tr A\,\tr C}{n}
 +\{\tr(C_0UA_0U^*):U\in SU(n)\}.
\]
Therefore, the $C$-numerical range of $A$ is star-shaped about
$(\tr A)(\tr C)/n$, which recovers the theorem of Cheung and Tsing
\cite{CheungTsing} without requiring a separate matrix-theoretic analysis of boundary points.

\section{Star-shapedness for representations of compact semisimple Lie groups}
The geometric principle in Proposition~\ref{prop:geometry} is not restricted to adjoint
representations. We record an extension to finite-dimensional complex representations.

Galindo, de la Harpe, and Vust \cite{GalindoHarpeVust} asked whether orbits
of compact connected irreducible complex linear groups meet every complex
hyperplane through zero. Gichev \cite[Corollary~1]{Gichev} answered this
affirmatively. He also proved a codimension-two orbit-intersection theorem
for real representations of compact connected semisimple groups whose
matrix coefficients have a common nonzero Laplace eigenvalue
\cite[Corollary~2]{Gichev}. These results give zero-inclusion for the
corresponding scalar orbit images. Here Proposition~\ref{prop:geometry}
yields the stronger radial star-shapedness property; a fixed component
translates the star center.

Let $K$ be a connected compact semisimple Lie group with Lie algebra $\kalg$,
equipped with the bi-invariant metric induced by the negative Killing form.
Let $E_1,\ldots,E_d$ be an orthonormal basis of $\kalg$, and let
$\rho:K\to\operatorname{GL}(V)$ be a finite-dimensional complex representation.
Its Casimir operator is
\[
 C_\rho:=-\sum_{j=1}^d d\rho(E_j)^2.
\]
This operator is independent of the orthonormal basis and commutes
with $\rho(K)$. Since $K$ is compact, we may choose a $K$-invariant
inner product on $V$. Each $d\rho(E_j)$ is then skew-adjoint, and hence
\[
 \langle C_\rho w,w\rangle
 =\sum_{j=1}^d\|d\rho(E_j)w\|^2\qquad(w\in V).
\]
Thus $C_\rho$ is positive semidefinite, and $C_\rho w=0$ holds precisely
when $d\rho(E_j)w=0$ for all $j$. Since $K$ is connected, this is
equivalent to $w$ being fixed by $K$. Consequently,
\[
 \ker C_\rho=V^K.
\]
Let
\[
 P_0:=\int_K\rho(k)\,dk
\]
denote the orthogonal projection onto $V^K$, where Haar measure is
normalized to have total mass one. In particular, $P_0$ commutes
with $\rho(K)$.

\begin{theorem}[Star-shapedness for a common Casimir eigenspace]\label{thm:representations}
Let $K$, $\rho$, and $P_0$ be as above. Suppose that $v\in V$ satisfies
\[
 v-P_0v\in\ker(C_\rho-cI)
\]
for some $c>0$. Then, for every complex-linear functional $f\in V^*$,
the set
\[
 W_f^\rho(v):=\{f(\rho(k)v):k\in K\}
\]
is star-shaped about $f(P_0v)$. More precisely,
\[
 tW_f^\rho(v)+(1-t)f(P_0v)\subseteq W_f^\rho(v),
 \qquad 0\leq t\leq1.
\]
\end{theorem}
\begin{proof}
Write $v=v_0+v_c$, where $v_0=P_0v$ and
$v_c=v-P_0v\in\ker(C_\rho-cI)$. Since $\rho(k)v_0=v_0$ for every
$k\in K$, we have
\[
 W_f^\rho(v)=f(v_0)+\{f(\rho(k)v_c):k\in K\}.
\]
If $v_c=0$, the result is immediate. Otherwise, let
$V_c=\ker(C_\rho-cI)$ and define the smooth orbit map
$a:K\to V_c\setminus\{0\}$ by $a(k)=\rho(k)v_c$.
The subspace $V_c$ is $K$-invariant because $C_\rho$ commutes with
$\rho(K)$, and $a$ never vanishes because $\rho(k)$ is invertible.
For the left-invariant vector fields corresponding to $E_1,\ldots,E_d$,
the bi-invariant metric gives
\[
 \Delta a(k)
 =\rho(k)\sum_{j=1}^d d\rho(E_j)^2v_c
 =-\rho(k)C_\rho v_c=-c\,a(k).
\]
Thus $\Delta(\ell\circ a)=-c(\ell\circ a)$ for every
$\ell\in V_c^*$. Since $K$ is connected, compact, and semisimple,
$H^1(K;\R)=0$. Lemma~\ref{lem:separation} therefore gives the
two-sided separation property. Proposition~\ref{prop:geometry},
applied to $a$ and $f|_{V_c}$, shows that
$\{f(\rho(k)v_c):k\in K\}$ is star-shaped about zero.
Translation by $f(v_0)$ gives the result.
\end{proof}

\begin{corollary}[Irreducible representations]\label{cor:irreducible}
If $\rho$ is a nontrivial irreducible finite-dimensional complex
representation of a connected compact semisimple Lie group $K$, then
$W_f^\rho(v)$ is star-shaped about zero for all $v\in V$ and $f\in V^*$.
\end{corollary}
\begin{proof}
By Schur's lemma, $C_\rho=cI$ for some scalar $c$. The positivity
identity above shows that $c>0$, since an irreducible representation
with $c=0$ would be trivial. Hence $P_0=0$, and the assertion follows
from Theorem~\ref{thm:representations}.
\end{proof}

\begin{remark}
Theorem~\ref{thm:representations} applies to a particular orbit vector
$v$ even when the representation has other constituents with different
positive Casimir eigenvalues. In particular, if the spectrum of
$C_\rho$ is contained in $\{0,c\}$ for some $c>0$, then every
$v\in V$ satisfies its hypothesis, and
\[
 C_\rho=c(I-P_0).
\]
This includes direct sums of irreducible representations with the same
positive Casimir eigenvalue, together with any trivial summands. It also
includes the adjoint representation in Theorem~\ref{thm:semisimple}:
by Lemma~\ref{lem:casimir}, its Casimir operator is the identity,
even when the semisimple Lie algebra has several simple ideals.
Moreover, the star center is the Haar average of the values,
\[
 f(P_0v)=\int_K f(\rho(k)v)\,dk.
\]
For general orbit vectors involving several distinct positive Casimir
eigenvalues, the present argument does not give a conclusion.
\end{remark}

\section*{Acknowledgment}
T.-Y. Tam thanks Y. Luan for his warm hospitality during his visit to the School of Mathematics, Shandong University, in May 2026.

\section*{AI disclosure}
During the preparation of this work, the authors used ChatGPT to assist with language editing, exposition, and exploration of the representation-theoretic extension in Section~5. The authors are responsible for verifying the mathematical arguments and the final manuscript.

\section*{Funding}
Y. Luan is supported by the Shandong Provincial Natural Science Foundation (Grant No. ZR2025MS36).


\begin{thebibliography}{99}
\bibitem{Aronszajn} N. Aronszajn, A unique continuation theorem for solutions of elliptic partial differential equations or inequalities of second order, \emph{J. Math. Pures Appl.} (9) \textbf{36} (1957), 235--249.
\bibitem{CheungTam} W.-S. Cheung and T.-Y. Tam, Star-shapedness and $K$-orbits in complex semisimple Lie algebras, \emph{Canad. Math. Bull.} \textbf{54} (2011), no.~1, 44--55.
\bibitem{CheungTsing} W.-S. Cheung and N.-K. Tsing, The $C$-numerical range of matrices is star-shaped, \emph{Linear Multilinear Algebra} \textbf{41} (1996), no.~3, 245--250.
\bibitem{DjokovicTam} D. \v{Z}. \DJ okovi\'c and T.-Y. Tam, Some questions about semisimple Lie groups originating in matrix theory, \emph{Canad. Math. Bull.} \textbf{46} (2003), no.~3, 332--343.
\bibitem{GalindoHarpeVust} J. Galindo, P. de la Harpe, and T. Vust, Two observations on irreducible representations of groups, \emph{J. Lie Theory} \textbf{12} (2002), no.~2, 535--538.
\bibitem{Gichev} V. M. Gichev, A note on common zeroes of Laplace--Beltrami eigenfunctions, \emph{Ann. Global Anal. Geom.} \textbf{26} (2004), 201--208.
\bibitem{Hirsch} M. W. Hirsch, \emph{Differential Topology}, Springer, 1976.
\bibitem{Knapp} A. W. Knapp, \emph{Lie Groups Beyond an Introduction}, 2nd ed., Birkh\"auser, 2002.
\bibitem{Tam2001} T.-Y. Tam, On the shape of numerical ranges associated with Lie groups, \emph{Taiwanese J. Math.} \textbf{5} (2001), 497--506.
\bibitem{TamCheung} T.-Y. Tam and W.-S. Cheung, The $K$-orbit of a normal element in a complex semisimple Lie algebra, \emph{Pacific J. Math.} \textbf{238} (2008), 387--398.
\end{thebibliography}
\end{document}